\documentclass[12pt]{amsart}
\usepackage{amsfonts, amsmath, amsthm}

\usepackage{graphicx}
\usepackage{epstopdf}
\usepackage{psfrag}

\newtheorem{pro}{Proposition}[section]
\newtheorem{thm}[pro]{Theorem}
\newtheorem{lem}[pro]{Lemma}

\newtheorem{cor}[pro]{Corollary}

\theoremstyle{definition}
\newtheorem{dfn}[pro]{Definition}

\theoremstyle{remark}

\newcommand{\CC}{\mathcal C}

\newcommand{\bdy}{\partial}
\newcommand{\BB}{\mathcal B}
\newcommand{\EE}{\mathcal E}
\newcommand{\DD}{\mathcal D}

\newcommand{\TT}{\mathcal T}
\newcommand{\KK}{\mathcal K}

\newcommand{\plex}[1]{\ensuremath{[{#1}]}}

\title{Almost Normal Surfaces with boundary}
\date{\today}
\author{David Bachman}
\author{Ryan Derby-Talbot}
\author{Eric Sedgwick}

\begin{document}
\begin{abstract}
We show that a strongly irreducible and $\bdy$-strongly irreducible surface can be isotoped to be almost normal in a triangulated $3$-manifold. 
\end{abstract}
\maketitle

\section{Introduction}

Throughout this paper $M$ will denote a compact, orientable, irreducible 3-manifold with (possibly empty) incompressible boundary admitting a triangulation $\TT$, and $H$ will represent a connected, 2-sided, orientable surface properly embedded in $M$ and not contained in a ball. A surface is said to be {\it strongly irreducible} if it has compressing disks for each side, and every pair of compressing disks on opposite sides intersects. A surface is {\it $\bdy$-strongly irreducible} if it has a compressing disk or $\bdy$-compressing disk for each side, and every pair on opposite sides intersects. In \cite{Bachman4} it is shown that every non-peripheral, strongly irreducible surface can be $\bdy$-compressed (possibly zero times) to a surface that is either also $\bdy$-strongly irreducible or is essential, and hence normal. (The case that $\partial M$ is a torus is also considered in \cite{BSS}.) In the following theorem we address the former case:

\begin{thm}
\label{t:main1}
Let $H$ be a strongly irreducible and $\bdy$-strongly irreducible surface in a 3-manifold $M$ with triangulation $\TT$.  Then $H$ is isotopic to a surface that is almost normal with respect to $\TT$.
\end{thm}

When $M$ is closed, this result is due independently to Stocking \cite{Stocking} and Rubinstein \cite{Rubinstein}. Related results when $M$ has boundary have been established by Bachman \cite{Bachman1},  Coward \cite{Coward}, Rieck-Sedgwick \cite{Rieck-Sedgwick}, Wilson \cite{Wilson}, and Johnson \cite{Johnson}. 

We have several motivations. First, Theorem~\ref{t:main1} has both fewer hypotheses and a stronger conclusion than previous results. Secondly, the theorem is useful for attacking various questions relating Dehn surgery and amalgamation along tori to the set of incompressible surfaces and strongly irreducible Heegaard surfaces in a given 3-manifold \cite{BDTS1}, \cite{BDTS2}. Lastly, we present here what we consider to be a more systematic and modern approach to this topic, one that can be generalized to arbitrary {\it topological index} (this program will be completed by the first author in \cite{Bachman3}, see \cite{Bachman2} for definitions).

Our approach exploits properties of disk complexes, defined in Section 2, associated with the surface $H$. Sections 3 and 4 show how strong irreducibility and $\partial$-strong irreducibility can be preserved (under isotopy) as we pass to disk complexes relative to the 1-skeleton and 2-skeleton, respectively. In Section 5, we analyze the structure of the resulting surface in each tetrahedron, concluding that it is almost normal. 

\section{Compressing Disks and Disk Complexes}

In this section we give basic definitions and facts regarding compressing disks and associated disk complexes.  Recall that $H$ is a connected, 2-sided, orientable surface properly embedded, but not contained in a ball, in a compact, orientable, irreducible 3-manifold $M$ with (possibly empty) incompressible boundary and triangulation $\TT$.

\begin{dfn}
A properly embedded loop or arc $\alpha$ on $H$ is said to be {\it inessential} if it cuts off (i.e.~it bounds or cobounds, with a subarc of $\bdy H$) a subdisk of $H$, and {\it essential} otherwise.
\end{dfn}

We define three types of compressing disks, extending the standard definition to account for the presence of a given complex $\KK$.   We are primarily concerned with the cases when $\KK = \emptyset$ or $\TT^i$, the $i$-skeleton of our triangulation $\TT$.

\begin{dfn}
A {\it compressing disk} $C$ for $(H,\KK)$ is a disk embedded in $M-N(\KK)$ so that $C \cap H=\bdy C$ is an essential loop in $H-N(\KK)$.

Let $\CC(H,\KK)$ denote the set of  compressing disks for $(H,\KK)$.
\end{dfn}

\begin{dfn}
A {\it $\bdy$-compressing disk}  $B$ for $(H,\KK)$ is a disk embedded in $M-N(\KK)$ so that $\bdy B=\alpha \cup \beta$, where $\alpha=B \cap H$ is an essential arc on $H-N(\KK)$ and $B \cap \bdy M=\beta$.

Let $\BB(H,\KK)$ denote the set of  $\bdy$-compressing disks for $(H,\KK)$.
\end{dfn}

\begin{dfn} 
An {\it edge-compressing disk} $E$ for $(H,\KK)$ is a disk embedded in $M$ so that $\bdy E =\alpha \cup \beta$, where $\alpha= E  \cap H$ is an arc on $H$ and $\beta = E \cap \KK \subset e$, for some edge $e$ of $\KK$.

Let $\EE(H,\KK)$ denote the set of  edge-compressing disks for $(H,\KK)$. Figure~\ref{fig:compression_types} illustrates each type of compressing disk. 
\end{dfn}

\begin{figure}[h]
\psfrag{H}{\small $H$}
\psfrag{D}{\small $D$}
\psfrag{T}{\small $\mathcal T^1$}
\psfrag{p}{\small $\partial M$}
   \centering
   \includegraphics[height=3in]{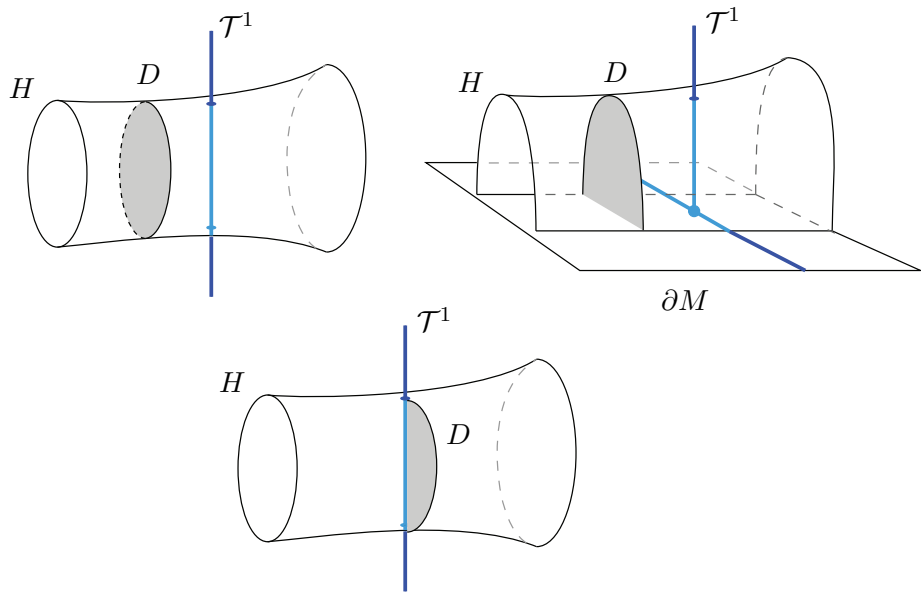} 
   \caption{Clockwise from upper left: $D$ is a compressing disk, $\partial$-compressing disk, and edge-compressing disk for $(H, \mathcal T^1)$.}
   \label{fig:compression_types}
\end{figure}

\begin{dfn}
\label{dCompression}
Suppose $D$ is a compressing disk, $\bdy$-compressing disk, or edge-compressing disk for $(H,\KK)$.    We construct a surface $H/D$, which is said to have been obtained by {\it surgering}  along $D$, as follows. Let $M(H)$ denote the manifold obtained from $M$ by cutting open along $H$. Let $N(D)$ denote a neighborhood of $D$ in $M(H)$. Construct the surface $H'$ by removing $N(D) \cap H$ from $H$ and replacing it with the frontier of $N(D)$ in $M(H) $. The surface $H/D$ is then obtained from $H'$ by discarding any component that lies in a ball.
\end{dfn}

\begin{dfn}
We say $H$ is {\it compressible} or {\it $\bdy$-compressible} if $\CC(H,\emptyset)$ or $\BB(H,\emptyset)$, respectively, is non-empty, and {\it incompressible} or {\it $\bdy$-incompressible} otherwise.  If $D$ is a compressing or $\bdy$-compressing disk for $H$ then $H/D$ is said to have been obtained by {\it compressing} or {\it $\bdy$-compressing along $D$}.
\end{dfn}

\begin{dfn}
We denote the set $\CC(H,\KK) \cup \BB(H,\KK)$ as $\CC \BB (H,\KK)$ and the set $\CC(H,\KK) \cup \EE(H,\KK)$ as $\CC \EE(H,\KK)$.
\end{dfn}

\begin{dfn}
We say $H$ is {\it strongly irreducible} if there are compressing disks in $\CC(H,\emptyset)$ on opposite sides of $H$, and any such pair of compressing disks has non-empty intersection. We say $H$ is {\it $\bdy$-strongly irreducible} if there are disks in $\CC \BB(H,\emptyset)$ on opposite sides of $H$, and any such pair of disks has non-empty intersection.
\end{dfn}

For the remainder of this section, let $\DD(H,\KK)$ denote either $\CC(H,\KK)$, $\CC \BB(H,\KK)$ or  $\CC \EE(H,\KK)$. We define an equivalence relation on disks as follows:

\begin{dfn}
Disks $D,D' \in \DD(H,\KK)$ are {\it equivalent}, $D \sim D'$, if $D$ and $D'$ are isotopic in $M$ through disks in $\DD(H,\KK)$.
\end{dfn}

We now distinguish three special types of pairs of disks:

\begin{dfn}
\label{dDiskPairs}
Suppose that $D,D' \in \DD(H,\KK)$.  We say that the pair $(D,D')$ is a

\begin{enumerate}
\item {\it disjoint pair} if $D \cap D' = \emptyset$.
\item {\it canceling pair} if $D \cap D'$ is a single point on an edge of $\KK$.
\item {\it coincident pair} if $D \cap D' = D \cap e = D' \cap e$ for an edge $e$ of $\KK$.
\end{enumerate}
\end{dfn}

Since only edge-compressing disks meet $\KK$, and then only in edges, a canceling pair is a pair of edge-compressing disks for opposite sides of $H$ whereas a coincident pair is a pair of edge-compressing disks for the same side of $H$.

\begin{dfn}
\label{dDiskComplex}
The {\it disk complex} $\plex{\DD(H,\KK)}$ is the graph whose:

\begin{enumerate}
\item vertices correspond to equivalence classes $[D]$, where $D \in \DD(H,\KK)$.
\item edges correspond to pairs of equivalence classes $([D],[D'])$, where $D$ and $D'$ are disjoint away from a neighborhood of the 1-skeleton, i.e.~$(D,D')$ is a disjoint, canceling, or coincident pair. 
\end{enumerate}

Let  $\plex{\DD(H)} =  \plex{\DD(H,\emptyset)}$.
\end{dfn}

The following is a weaker version (in a broader context) of a result of McCullough \cite{McCullough}. See also Theorem 4.2 of \cite{cho}.

\begin{lem}
\label{l:OneSideConnected}
Let $\DD_+(H,\KK)$ be the subset of $\DD(H,\KK)$ consisting of disks on one side of $H$, denoted $H_+$. Then $\plex{\DD_+(H,\KK)}$ is either empty or has one component.
\end{lem}

\begin{proof}
Suppose to the contrary that $\plex{\DD_+(H,\KK)}$ has 2 or more components.   Then choose disks $D$ and $D'$ that: 1) represent vertices in distinct components; 2) intersect transversally; and 3) minimize the total number of curves of intersection $|D \cap D'|$, subject to 1) and 2).   Note that $D \cap D' \neq \emptyset$, for otherwise $[D]$ and $[D']$ lie in the same component.

If the interiors of $D$ and $D'$ meet in a closed curve, choose an innermost such on $D$.   Then a cut and paste operation on $D'$ at that curve, along with a slight isotopy, produces a disk $D''$ disjoint from $D'$, hence in the same component, that meets $D$ transversally but in fewer curves.   This is a contradiction.

If  $D \cap D'$ contains an arc not contained in the boundary of the disks, let $\delta$ be an outermost such arc on $D$.  Cut, paste and slightly isotope $D'$ along $\delta$ to produce disks $D''$ and $D'''$, each of which meets $D$ in fewer curves than $D'$, and at least one of which is guaranteed to meet $H-N(\KK)$ in an essential arc or loop. Assume this is true of $D''$. Note that $D''$ may be in $\CC(H,\KK)$, even if neither $D$ nor $D'$ were. It is also worth noting that this operation can be performed regardless of whether $\delta$ has both endpoints on $H$, or one or more endpoints on $\partial M$ or an edge $e$.   In the latter  cases, both $D$ and $D'$ must have the same type, $\BB$ or $\EE$, as does the essential disk produced. In any case, $D''$ becomes a disk disjoint from $D'$, hence in the same component, which is a contradiction since it meets $D$ fewer times than $D'$. 

The only remaining case is that $D$ and $D'$ are a pair of edge-compressing disks for the same side of $H$ whose interiors are disjoint but which meet along an edge $e$ of $\KK$.  This describes a coincident pair, thus an edge joins $[D]$ and $[D']$.
\end{proof}

Let $\DD_-(H,\KK)$ now denote the subset of $\DD(H,\KK)$ consisting of those disks on the opposite side of $H$ as the disks in $\DD_+(H,\KK)$. Thus, $\plex{\DD(H,\KK)}$ is the union of $\plex{\DD_+(H,\KK)}$ and $\plex{\DD_-(H,\KK)}$ along with any edges between them.   Applying the previous lemma once to each side yields:

\begin{cor} Let $H \subset M$ be a separating surface that meets a complex $\KK$ transversally.   Then $\plex{\DD(H,\KK)}$ has either 0, 1 or 2 components.
\end{cor}

By definition, $\plex{\CC(H,\KK)} = \emptyset$ means $H$ is incompressible. If $\plex{\CC_+(H,\KK)}$ and $\plex{\CC_-(H,\KK)}$ are non-empty, $\plex{\CC(H,\KK)}$ will be connected if and only if there is a pair of compressing disks, one for each side, that are disjoint. The same holds when we consider the union of compressing and $\bdy$-compressing disks for $H$.

\begin{cor} 
\label{c:SI}
Let $H \subset M$ be a separating surface.  Then,
\begin{enumerate}
\item $H$ is {\it incompressible} $\iff$ $\plex{\CC(H)}=\emptyset$.
\item $H$ is {\it incompressible} and $\bdy$-incompressible $\iff$ $\plex{\CC \BB(H)}=\emptyset$.
\item $H$ is {\it strongly irreducible} $\iff$ $\plex{\CC(H)}$ is disconnected.
\item $H$ is $\bdy$-{\it strongly irreducible} $\iff$ $\plex{\CC \BB(H)}$ is disconnected.
\end{enumerate}
\end{cor}

To prove Theorem~\ref{t:main1}, we proceed from here in the following manner. Assume that $H$ is a surface for which $\plex{\CC(H)}$ and $\plex{\CC \BB(H)}$ are both disconnected. In Sections 3 and 4 we show that isotoping $H$, first with respect to the 1-skeleton, and then with respect to the 2-skeleton, makes an associated sequence of disk complexes each disconnected. We then use this in Section 5 to conclude that $H$ is almost normal. The flow chart in Figure~\ref{f:flow_chart} shows how the key results of the different sections of the paper link together to prove Theorem~\ref{t:main1}.

\begin{figure}[h]

\psfrag{a}{\bf Proving Theorem ~\ref{t:main1}:}

\psfrag{b}{$H$ is strongly irreducible and $\partial$-strongly irreducible}

\psfrag{c}{\scriptsize Corollary \ref{c:SI}}

\psfrag{d}{$\plex{\CC(H)}$ and $\plex{\CC \BB(H)}$ are disconnected}

\psfrag{f}{\small \begin{tabular}{c} {\scriptsize Proposition \ref{p1Skeleton}} \\ {\tiny (isotope $H$)} \end{tabular}}

\psfrag{g}{$\plex{\CC \EE(H,\mathcal T^1)}$ is disconnected} 

\psfrag{i}{\small \begin{tabular}{l} {\scriptsize Proposition \ref{p:CE1toCE2}} \\ {\tiny (isotope $H$ rel $\mathcal T^1$)} \end{tabular}}

\psfrag{j}{$\plex{\CC \EE(H,\mathcal T^2)}$ is disconnected}

\psfrag{k}{\scriptsize Proposition \ref{p:normalOrAlmostNormal}}

\psfrag{l}{$H$ is almost normal}

\hspace{-4.5cm}
\includegraphics[height=3.5in]{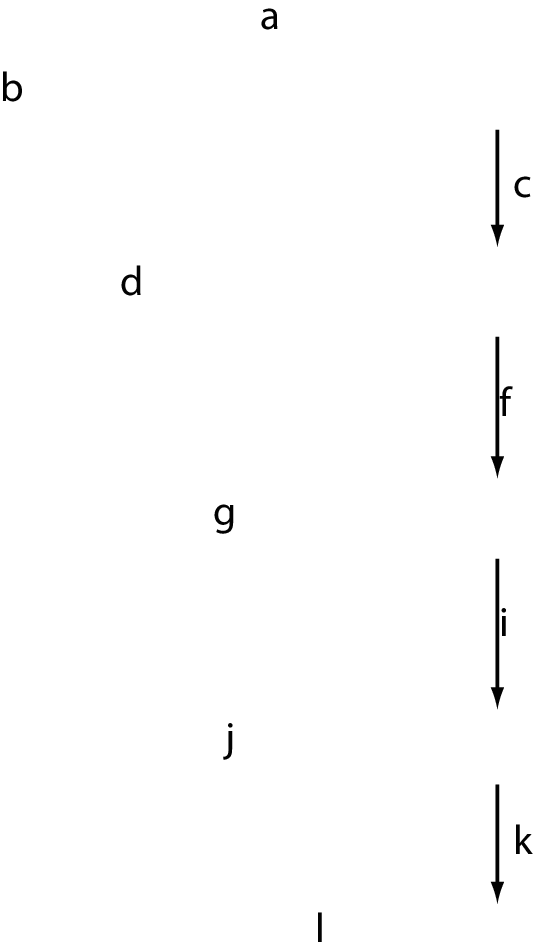}
\label{f:flow_chart}
\caption{How the main results of each section link together to prove Theorem \ref{t:main1}.}

\end{figure}

\section{The 1-skeleton: $\plex{\CC(H)}$ and $\plex{\CC \BB(H)}$ disconnected implies $\plex{\CC \EE(H,\TT^1)}$ is disconnected}

In this section we prove that $H$ can be isotoped so that it is ``strongly irreducible'' with respect to the 1-skeleton.

\begin{pro}
\label{p1Skeleton}
Suppose $\plex{\CC(H)}$ and $\plex{\CC \BB(H)}$ are disconnected. Then  $H$ may be isotoped so that $\plex{\CC \EE(H,\TT^1)}$ is disconnected.
\end{pro}

The remainder of this section is a sequence of lemmas proving this proposition.  The main idea is to find a suitable position for $H$ as a {\it thick level} in a {\it compressing sequence}, which we will introduce below.

Note that some compressing disks for $(H,\TT^1)$ are compressing disks for $H = (H,\emptyset)$ and some are not. We make this distinction in the following definition.

\begin{dfn}
Let $D \in \CC \EE(H,\TT^1)$. If $D \in \CC \BB(H,\emptyset)$ then we say $D$ is {\it honest}. Otherwise, we say $D$ is {\it dishonest}. (See Figure~\ref{fig:honest_and_dishonest_compressing_disk}.)
\end{dfn}

\begin{figure}[h]
\psfrag{H}{$H$}
\psfrag{p}{$\TT^1$}
\psfrag{D}{$D$}
   \centering
   \includegraphics[width=4.5in]{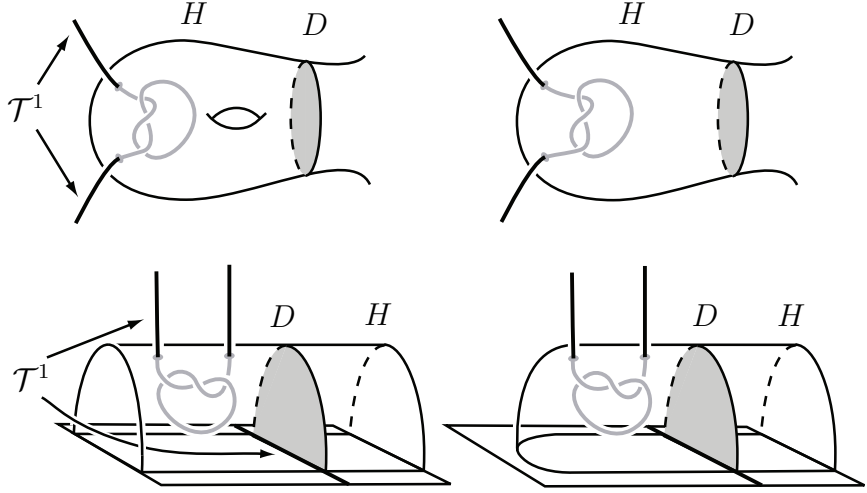}
   \caption{Starting from the top left and going clockwise, $D$ is: an honest compressing disk, a dishonest compressing disk, a dishonest edge-compressing disk incident to a boundary edge of $\TT^1$, and an honest edge-compressing disk incident to a boundary edge of $\TT^1$.}
   \label{fig:honest_and_dishonest_compressing_disk}
\end{figure}

Note that edge-compressing disks incident to interior edges of $\TT^1$ are dishonest, and if $\bdy M$ is incompressible and $|\partial H \cap \TT^1|$ is minimized, edge-compressing disks incident to boundary edges are honest.

We now define a complexity on $H$ (rel $\TT^1$), and show that both honest and dishonest compressions decrease complexity.

\begin{dfn}
If $H$ is empty, then we define the {\it width} of $H$, $w(H)$, to be $\emptyset$. If $H$ is connected, then the {\it width} of $H$, $w(H)$, is the pair, $(-\chi(H), |\TT^1 \cap H|)$. If $H$ is disconnected, then its {\it width} is the ordered set of the widths of its components, where we include repeated pairs and the ordering is non-increasing. Comparisons are made lexicographically at all levels.
\end{dfn}

Recall that $H/D$ denotes the surface obtained by surgering $H$ along $D$ and discarding any components that lie in a ball.
\begin{lem}
\label{l:widthchange}
Let $D \in \CC \EE(H,\TT^1)$. Then $w(H/D)<w(H)$.  If $D$ is dishonest, then $H/D$ is isotopic to $H$ in $M$.
\end{lem}

\begin{proof}
If $D$ is an edge-compressing disk meeting an interior edge of $\TT^1$, then $H/D$ is isotopic in $M$ to $H$, but meets $\TT^1$ two fewer times. 

When $D$ is either a dishonest compressing disk, or a dishonest edge-compressing disk incident to a boundary edge of $\TT^1$, then surgering along $D$ produces a surface with two components. Because $M$ is irreducible and $\bdy M$ is incompressible, one of these components is isotopic in $M$ to $H$, and the other is either a disk or sphere meeting $\TT^1$.   As the disk or sphere lies in a ball, we discard it to obtain $H/D$,  a surface isotopic to $H$ in $M$ but meeting $\TT^1$ in fewer points.

Suppose then that $D$ is honest, so that $D \cap H$ is essential in $H$. If, furthermore, $\bdy D$ does not separate $H$, then $-\chi(H/D)$ is less than $-\chi(H)$, and hence the width is also less. If, on the other hand, $\bdy D$ separates $H$ then $H/D$ is disconnected, and both components have smaller negative Euler characteristic. Hence, again width has decreased.
\end{proof}

\begin{dfn}
A {\it compressing sequence} is a sequence $\{H_i\}$ of (possibly empty) 2-sided, embedded surfaces such that for each $i$ either (i) $H_i$ is isotopic, relative to $\TT^1$, to $H_{i-1}$; (ii) $H_i=H_{i-1}/D$, for some $D \in \plex{\CC \EE(H_{i-1},\TT^1)}$; or (iii) $H_i/D=H_{i-1}$, for some $D \in \plex{\CC \EE(H_{i},\TT^1)}$.
\end{dfn}

We have allowed successive terms to be isotopic as a notational convenience, so terms of maximal complexity are not necessarily isolated.

\begin{dfn}
Let $\{H_i\}$ be a compressing sequence. A  subsequence  $\{H_j,\dots,H_k\}$ is said to be a {\it plateau} if all terms are isotopic relative to $\TT^1$ (and hence have equal width), $w(H_j)>w(H_{j-1})$ and $w(H_k)>w(H_{k+1})$.   Every surface on a plateau is referred to as a {\it thick level}.
\end{dfn}

\begin{dfn}
Let $\{H_i\}$ be a compressing sequence. Then the {\it size} of the entire sequence is the set
\[\{w(H_i) \ | \  H_i \mbox{ is on a plateau}\},\]
where repetitions are included, and the ordering is non-increasing. Two such sets are compared lexicographically.
\end{dfn}

We now aim to build a compressing sequence realizing $H$ as a thick level with disconnected disk complex with respect to $\mathcal T^1$.   The following lemma will be used to first normalize the boundary of $H$.

\begin{lem}
\label{l:EssentialBoundary}
If $H$ is connected and $\plex{\CC(H)}$ is disconnected, then  $\partial H$ consists entirely of curves essential in $\partial M$.
\end{lem}

\begin{proof}
Since $\plex{\CC(H)}$ is disconnected, it must be non-empty. Hence, the surface $H$ must not be a disk. Hence, every component of $\bdy H$ must be essential on $H$. Thus, if some component of $\bdy H$ is inessential on $\bdy M$, then an innermost one (on $\bdy M$) bounds a subdisk $C$ of $\bdy M$ that is a compressing disk for $H$.  Furthermore, as $C$ is in $\partial M$, it is disjoint from any other compressing disk for $H$.  This would imply that $\plex{\CC(H)}$ is connected, a contradiction.
\end{proof}

We now prove that the hypothesis of Proposition~\ref{p1Skeleton}, that $\plex{\CC(H)}$ is disconnected, also allows us to construct a compressing sequence in which all non-terminal surfaces are isotopic to $H$.

\begin{lem}
\label{lCompSequence}
Suppose $\plex{\CC(H)}$ is disconnected. Then there is a compressing sequence $\{H_i\}_{i=0}^{n}, n \geq 2$ satisfying the  following conditions:
\begin{enumerate}
	\item All non-terminal surfaces are isotopic to $H$ in $M$, and the boundaries of such surfaces meet $\TT^1$ minimally. 
	\item The initial element is obtained from the next, $H_0 = H_1/D_+$, by compressing along an honest disk, $D_+ \in \CC \EE(H_1,\TT^1)$.
	\item The final element is obtained from the penultimate one, $H_n = H_{n-1}/D_-$, by compressing along an honest disk, $D_- \in \CC \EE(H_{n-1},\TT^1)$.
	\item $D_+$ and $D_-$, as compressing or $\bdy$-compressing disks for $H$, represent vertices in different components of $\plex{\CC \BB(H)}$. (Recall that honest edge-compressing disks are $\bdy$-compressing disks).
\end{enumerate}
\end{lem}

\begin{proof}
First we isotope $H$ to minimize the intersection with $\TT^1 \cap \bdy M$. (Note that every loop of $\bdy H$ is essential by Lemma \ref{l:EssentialBoundary}, and hence the resulting loops still meet $\TT^1$.)

By assumption, we have compressing disks $D_+$ and $D_-$ for $H$, such that $[D_+]$ and $[D_-]$ are in different components of $\plex{\CC(H)}$.  We now use the disk $D_-$ to construct a compressing sequence. First, let $\{\alpha_i\}_{i=1}^n$ be a set of pairwise-disjoint arcs in $D_-$ connecting the points of $\TT^1 \cap D_-$ to points in $D_- \cap H$. The sequence $\{H_i\}_{i=0}^n$ is then defined by induction as follows. First, let $H_0=H$. Let $M(H_i)$ denote the manifold obtained from $M$ by cutting open along $H_i$. Let $N(\alpha_i)$ denote a neighborhood of $\alpha_i$ in $M(H_i)$. Construct the surface $H_{i+1}$ by removing $N(\alpha_i) \cap H_i$ from $H_i$ and replacing it with the frontier of $N(\alpha_i)$ in $M(H_i)$. Note that a co-core of $N(\alpha_i)$ is then a dishonest compressing disk for $(H_{i+1},\TT^1)$, and thus the sequence $\{H_i\}$ is a dishonest compressing sequence (successive terms are related by dishonest compression). Note also that since $D_-$ is in the interior of $M$, all surfaces $H_i$ thus constructed have the same boundary. 

By a similar construction, we can use the disk $D_+$ to construct a sequence of surfaces $\{H_i\}_{i=-m}^0$, where $H_0=H$. Putting these sequences together then gives a dishonest compressing sequence. Let $H_{-m-1}=H/D_+$ and $H_{n+1}=H/D_-$. These surfaces are obtained by honest compression, and hence are of lower width than every other surface in the sequence. 
\end{proof}

\begin{dfn}
\label{d:PosNeg}
Suppose $\{H_i\}$ is a compressing sequence satisfying the conclusion of Lemma \ref{lCompSequence}. Let $D$ be an honest element of $\CC \EE(H,\TT^1)$. We will say $D$ is {\it positive} if it belongs to the same component of $\plex{\CC \BB(H)}$ as the honest disk $D_+$ from condition (2) of Lemma \ref{lCompSequence}, and {\it negative} if it belongs to the same component as the disk $D_-$ from condition (3).
\end{dfn}

Given disks $D$ and $E$ such that $[D]$ and $[E]$ span an edge of $\plex{\CC \EE(H,\TT^1)}$, we now define a surface $H/DE$. Roughly speaking, this surface will be obtained by simultaneous surgery along both $D$ and $E$. 

\begin{lem}
\label{ledgeCompression}
Suppose that $\plex{\CC \BB(H)}$ is disconnected and that  $|\partial H \cap \TT^1|$ is minimized.   Then given $D,E \in \CC \EE(H,\TT^1)$ that are either disjoint or form a canceling or coincident pair, there is a surface $H/DE$ such that $\{H/D, H/DE, H/E\}$ is a compressing sequence with  $w(H/DE) < w(H)$.
\end{lem}

\begin{proof}
If $D$ and $E$ are a disjoint pair, then they have disjoint neighborhoods and we let $H/DE$ be the surface obtained by simultaneously surgering along $D$ and $E$ and discarding any components that lie in a ball.  (Compare Definition \ref{dCompression}).   Now either $E \in \CC \EE(H/D,\TT^1)$, its boundary is trivial in $H/D$, or it meets $H$ in a component that was discarded when constructing $H/D$.     In the first case, $H/DE = (H/D)/E$ and in the latter two cases, compressing $H/D$ along $E$ has no effect, hence $H/DE$ is isotopic to $H/D$.   By symmetry, $H/DE$ is either isotopic to $H/E$ or $(H/E)/D$.   It follows that $\{H/D, H/DE, H/E\}$ is a compressing sequence with width $w(H/DE) < w(H)$.

If $D$ and $E$ are a canceling pair, then $D$ and $E$ are a pair of edge-compressing disks on opposite sides of $H$ meeting in a single point along an edge.    If the edge is interior, then both are dishonest and $H/D$ is isotopic to $H/E$, relative to $\TT^1$.   Define $H/DE$ to be $H/D$.    Then $\{H/D, H/DE, H/E\}$ is a compressing sequence of isotopic surfaces, all of equal width less than that of $H$.

If $D$ and $E$ are a canceling pair, and meet along a boundary edge of $\TT^1$, then, as $|\partial H \cap \TT^1|$ was minimized, both are honest and $D,E \in \BB(H)$. Now, $D$ and $E$ meet in a single point on $\partial H$, but after a slight isotopy while they are no longer edge-compressing disks, they are still $\bdy$-compressing disks on opposite sides of $H$.   Hence by Lemma \ref{l:OneSideConnected} $\plex{\CC \BB(H)}$ is connected, a contradiction.

If $D$ and $E$ are a coincident pair meeting an edge of $\TT^1$, then $D$ and $E$ are a pair of edge-compressing disks on the same side of $H$ and are coincident along a sub-arc of an edge.    In this case $D \cup E$ is a disk meeting $\TT^1$ in an arc.   Define $H/DE$ to be the surface obtained by surgering $H$ along $D \cup E$ and discarding any components that lie in a ball (in the case that $D$ and $E$ meet a boundary edge of $\TT^1$ the disk $D \cup E$ is not properly embedded in $M$, but surgering as in Definition~\ref{dCompression} produces the desired surface).  Note that $E'=E-N(D)$ is a compressing disk for $H/D$.  It is not trivial, because $D$ and $E$ represent different vertices.    Thus, $H/DE$ is isotopic to $(H/D)E'$.  By symmetry, $H/DE$ is isotopic to $(H/E)/D'$.   It follows that $\{H/D, H/DE, H/E\}$ is a compressing sequence with width $w(H/DE) < w(H)$.
\end{proof}

We use Lemma \ref{ledgeCompression} to obtain a compressing sequence from any path in the disk complex. 

\begin{dfn}
\label{lpathCS}
Suppose that $\plex{\CC \BB(H)}$ is disconnected and  $|\partial H \cap \TT^1|$ is minimized. Let $p=\{[D_i]\}_{i=0}^n$ be a path in $\plex{\CC \EE(H,\TT^1)}$.  Then we define $H/p$ to be the compressing sequence $$\{H/D_0,H/D_0D_1,H/D_1,H/D_1D_2,H/D_2,\dots, H/D_{n-1}D_n,H/D_n\}$$ such that for each $i$, $H/D_i D_{i+1}$ is the surface given by Lemma \ref{ledgeCompression}.
\end{dfn}

Proposition \ref{p1Skeleton} follows from Lemma \ref{lCompSequence} and the following:

\begin{lem} 
Let $H_i$ be a thick level in a minimal size compressing sequence that satisfies the conditions of Lemma \ref{lCompSequence}.   Then $\plex{\CC \EE(H_i,\TT^1)}$ is disconnected and $H_i$ is isotopic to $H$ in $M$.
\end{lem}

\begin{proof}
Since the widths of the initial and final terms are less than any intermediate term, by condition (1) in Lemma~\ref{lCompSequence} $H_i$ is isotopic to $H$ in $M$.   Since the size of the compressing sequence can be reduced by identifying successive isotopic surfaces on a plateau, we may assume that $H_i$ is the only surface on the plateau.

Let $D_i$ and $D_{i+1}$ be the disks for which $H_{i-1} = H_i/D_i$ and $H_{i+1}=H_i/D_{i+1}$.    By way of contradiction, suppose that $\plex{\CC \EE(H_i,\TT^1)}$ is connected, let $p$ be a path from $[D_i]$ to $[D_{i+1}]$, and $H_i/p$ the compressing sequence from $H_{i-1}$ to $H_{i+1}$ given by Definition \ref{lpathCS}.  Produce a new compressing sequence by replacing  the subsequence $\{H_{i-1},H_i,H_{i+1}\}$ in our original sequence with the compressing sequence $H_i/p$.   The new sequence is a compressing sequence with lower complexity than the original, because the thick level $H_i$ was eliminated and replaced with terms all with strictly lower complexity by Lemma \ref{l:widthchange}.  We will need to pass to a subsequence to satisfy conditions (1)--(4) of Lemma \ref{lCompSequence}.

By construction, each term of the new sequence either 
	\begin{itemize}
		\item is isotopic in $M$ to $H$,
		\item has the form $H_i/D$ or $H_i/DE$, where $H_i$ is isotopic to $H$ in $M$, and one or both of $D$ or $E$ is honest, or 
		\item has the form $H_i/DE$, where $D$ and $E$ are a coincident pair of edge-compressing disks. 
	\end{itemize}

We now label the terms in our sequence that are not isotopic to $H$. If such a term has the form $H_i/D$, where $D$ is honest, then label it with ``$+$" or ``$-$" according to whether $D$ is positive or negative. (Recall Definition \ref{d:PosNeg}.) When the term has the form $H_i/DE$ then there are two cases. If only one of $D$ or $E$ is honest, then as before assign the label ``$+$" or ``$-$" according to whether the disk is positive or negative. If both $D$ and $E$ are honest then note that even when they are both incident to $\bdy M$, they can be isotoped to be disjoint as elements of $\plex{\CC \BB(H)}$. Thus, they are in the same component of $\plex{\CC \BB(H)}$ and are hence either both positive, or both negative. Assign the label ``$+$" or ``$-$"  accordingly. Finally, when the term has the form $H_i/DE$, where $D$ and $E$ are a coincident pair of edge-compressing disks, then the disk $D \cup E$ is (isotopic to) an element of $\CC \BB(H)$. (Otherwise $H_i/DE$ would be isotopic to $H_i$.) Hence, we may assign a label based on whether or not the disk $D \cup E$ is positive or negative.

Note that no term labeled $+$ can be adjacent to a term labeled $-$.   Condition (1) of Lemma \ref{lCompSequence} guarantees this for the original sequence and the construction given by Definition \ref{lpathCS} guarantees it for the substituted sequence $H_i/p$, because successive terms are labeled with a pair of compressing disks that are joined by an edge, hence are in the same component when both are honest.

Note that the new sequence starts with a $+$ term and ends with a $-$ term.   Pass to the shortest subsequence with terminals of opposite labels.   Note that neither terminal is of the form $H_i/DE$ where $D$ and $E$ are both positive or both negative, because terminating instead at either $H_i/D$ or $H_i/E$ would provide a shorter subsequence. Thus, all intermediate terms are isotopic to $H$ in $M$.   This subsequence satisfies the conditions of Lemma \ref{lCompSequence} but has lower size, a contradiction.
\end{proof}

\section{The 2-skeleton: $\plex{\CC \EE(H,\TT^1)}$ disconnected implies $\plex{\CC \EE(H,\TT^2)}$ is disconnected. }

\begin{pro}
\label{p:CE1toCE2}
If $\plex{\CC \EE(H,\TT^1)}$ is disconnected, then $H$ may be isotoped, relative to $\TT^1$,  so that $\plex{\CC \EE(H,\TT^2)}$ is disconnected.
\end{pro}

The remainder of the section is a sequence of lemmas proving this proposition. To this end, we define:

\begin{dfn}
A disk $D$ is a {\it shadow disk} for $(H,\TT^1)$ if 
	\begin{enumerate}
		\item either $\bdy D$ is an essential loop on $H-N(\TT^1)$, or $\bdy D =\alpha \cup \beta$, where $\alpha$ is an arc on $H$ and $\beta \subset e$, for some edge $e$ of $\TT^1$, and
		\item the interior of $D$ meets $H - N(\TT^1)$ in a (possibly empty) collection of loops that are inessential in $H - N(\TT^1)$.
	\end{enumerate}
\end{dfn}

Note that if the interior of a shadow disk $D$ is disjoint from $H$, then it is either a compressing or an edge-compressing disk for $(H,\TT^1)$.

\begin{lem}
\label{d:ShadowImpliesCompressing}
Let $\delta$ be the boundary of a shadow disk.  Then $\delta$ is also the boundary of a compressing disk or edge-compressing disk for $(H,\TT^1)$.
\end{lem}

\begin{proof}
Among shadow disks $D$ with $\bdy D = \delta$ choose one that meets $H$ in the minimal number of loops.    If  there is at least one loop, let $\delta'$ be one that is innermost on $H - N(\TT^1)$.  By surgering $D$ along the disk in $H$ bounded by $\delta'$ we produce a disk with boundary $\delta$ but fewer intersections, a contradiction.
\end{proof}

\begin{dfn}
Suppose $D$ is a compressing or edge-compressing disk, and $D'$ is a shadow disk with the same boundary. Then we say $D'$ {\it is a shadow of} $D$.
\end{dfn}

We now proceed with the proof of Proposition \ref{p:CE1toCE2}. Choose compressing or edge-compressing disks $D_+,D_-$ in different components of  $\plex{\CC \EE(H,\TT^1)}$. Define an isotopy  $f: H \times I \to M$, relative to $\TT^1$, so that

      \begin{enumerate}
               \item $f(H,1/2)=H$.
               \item If $D_+$ is a compressing disk, then $f(H,0)$ is obtained from $H$ by shrinking $D_+$ until it is disjoint from $\TT^2$. If $D_+$ is an edge-compressing disk, then $f(0)$ is obtained from $H$ by shrinking $D_+$ into a neighborhood of an edge $e$ of $\TT^1$ (keeping $D_+ \cap \TT^1$ fixed), and then possibly twisting it around $e$ so that it meets $\TT^2$ in a subarc of $e$.
               \item The surface $f(H,1)$ is obtained from $H$ in the same manner, depending on whether $D_-$ is a compressing disk or an edge-compressing disk.
       \end{enumerate}
We may assume the isotopy $f$ is in general position with respect to $\TT^2-\TT^1$. Let $H_t=f(H,t)$.  Let $\{t_i\}$ denote the set of values so that $H_t$ is not transverse to $\TT^2-\TT^1$. 

Label the intervals $(t_i, t_{i+1})$ as follows.  Choose $t \in (t_i, t_{i+1})$. If there is a compressing or edge-compressing disk $D \in \CC \EE(H_t,\TT^1) =\CC \EE(H,\TT^1)$ for $H_t$ in the same component of $\plex{\CC \EE(H,\TT^1)}$ as $D_+$ ($D_-$), with a shadow $D'$ meeting $\TT^2$ in at most a sub arc of an edge, then label the interval ``$+$" (``$-$," resp.). By construction, the first interval is labeled ``$+$" and the last is labeled ``$-$". It follows that either

\begin{enumerate}
       \item There is a ``$+$" interval adjacent to a ``$-$" interval.
       \item There is an interval with both labels.
       \item There is an unlabeled interval.
\end{enumerate}

\begin{lem}
There cannot be an interval with just the label ``$+$" adjacent to an interval with just the label ``$-$".
\end{lem}

\begin{proof}
First note that passing through a center tangency between $H_t$ and $\TT^2$ (away from $\TT^1$) does not change labeling, as it merely creates or destroys a trivial loop which does not have an effect on shadow disks. So with no loss of generality suppose that a saddle tangency separates the two intervals.  Let $\sigma$ denote the singular curve in the face $\Delta$ of $\TT^2$ at the time of tangency. Then $\sigma$ is homeomorphic to either an ``$\times$," an ``$\alpha$" or an ``$\infty$."   Furthermore, in the case of an ``$\times$," two of the branches of the singular curve must meet the same edge of $\TT^1$. Let $H_+$ and $H_-$ denote the surface $H_t$, for $t$ in the intervals labeled ``$+$" and ``$-$", respectively, and denote the corresponding resolutions of $\sigma$ by $\sigma_+ \subset H_+ \cap \Delta$ and $\sigma_- \subset H_- \cap \Delta$.   Away from $N(\sigma)$ the intersection pattern of $H_+$ and $H_-$  with $\TT^2$ agree.

There are now several cases to consider for $\sigma$ (see Figure \ref{f:SigmaCurves}):
	\begin{enumerate}
		\item Either $\sigma_+$ or $\sigma _-$ contains a non-normal arc, say $\sigma_+$. Note that if $\sigma$ is an $\times$, then we are necessarily in this case.
		\item $\sigma$ is an $\infty$ curve. Then both $\sigma_+$ and $\sigma_-$ consists of loops. If all three of these loops are inessential on $H_+$, then the labeling could not have changed, as shadow disks would not have been affected. Hence, we may assume at least one of the loops, say a loop of $\sigma_+$, is essential on $H_+$.
		\item $\sigma$ is an $\alpha$ curve, and the arc components of $\sigma_+ \cup \sigma_-$ are normal. In this case there is a loop component  of $\sigma_+ \cup \sigma_-$, say in $\sigma_+$. If this loop is inessential on $H_+$, then again the labeling could not have changed. Hence again we conclude there is a loop of $\sigma_+$ that is essential on $H_+$.
	\end{enumerate}

\begin{figure}
\psfrag{1}{(1)}
\psfrag{2}{(2)}
\psfrag{3}{(3)}
\psfrag{+}{$\sigma_+$}
\psfrag{-}{$\sigma_-$}
\[\includegraphics[width=4.5in]{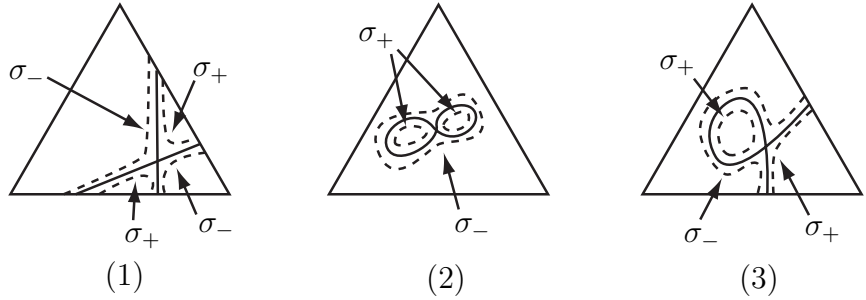}\]
\caption{The three cases for $\sigma$.}
\label{f:SigmaCurves}
\end{figure}

In all three cases a component of $\sigma_+$  bounds a subdisk  in a face that we can push off of $\TT^2$ to a shadow disk $D'$. By Lemma \ref{d:ShadowImpliesCompressing}, $D'$ is the shadow of a compressing or edge-compressing disk $D$. By assumption, $D$ must therefore be in the same component of $\plex{\CC \EE(H,\TT^1)}$ as $D_+$.

Note that $\bdy D=\bdy D'$ lies outside of a neighborhood of $\sigma$ on $H_+$. Therefore, $\bdy D'$ is also a curve on $H_-$ that meets $\TT^2$ in at most its endpoints. We conclude $D'$ is a shadow disk for $H_-$, and thus the interval containing $H_-$ is also labelled ``$+$," a contradiction.
\end{proof}

We are left with the possibility that there is either an unlabeled interval, or an interval with both labels. Let $t$ be in such an interval.

\begin{lem}
$H_t$ meets $\TT^2$ in a collection of normal arcs, and loops that are inessential on $H_t-N(\TT^1)$.
\end{lem}

\begin{proof}
If the lemma is false then $H_t$ meets some face in a non-normal arc, or a loop that is essential on $H_t-N(\TT^1)$.  An innermost such loop/outermost such arc bounds a subdisk of a face that is a shadow disk $D'$. A small isotopy of  $D'$ makes it meet $\TT^2$ in at most a subarc of an edge, and hence the interval containing $t$ cannot be unlabeled.

By Lemma \ref{d:ShadowImpliesCompressing}, $D'$ is a shadow of a compressing or edge-compressing disk $D$.  Without loss of generality, we assume $D$ is in the same component of $\plex{\CC \EE(H,\TT^1)}$ as $D_+$. Since the interval containing $t$ has both labels, there must be a compressing or edge-compressing disk $E$ for $H_t$, which has a shadow $E'$ meeting $\TT^2$ in at most a subarc of an edge, where $E \in \plex{\CC \EE(H,\TT^1)}$. Thus, $D \cap H_t$ meets $E \cap H_t$ away from $\TT^1$. However, as $D'$ is isotopic into $\TT^2$, it can be isotoped to miss $E'$ away from $\TT^1$, contradicting the fact that $\bdy D'=\bdy D$ and $\bdy E'=\bdy E$.
\end{proof}

\begin{lem}
\label{l:each_face_normal_arcs}
There is an isotopy of $H$, relative to $\TT^1$, so that $H$ meets each face of $\TT^2$ only in normal arcs.
\end{lem}

\begin{proof}
It follows from the prior two lemmas that $H$ can be isotoped relative to $\TT^1$ to a surface $H_t$ that  meets each face in a collection of normal arcs, and loops that are inessential in $H_t-N(\TT^1)$. Let $\delta$ be an inessential loop of intersection that bounds a disk $D$ that is innermost in a face.   Then by surgering $H$ along $D$ we produce a surface isotopic to $H_t$ relative to $\TT^1$ that has fewer inessential intersections. Proceeding in this way we may remove all inessential loops, and are thus left with a surface that meets $\TT^2$ in only normal arcs. 
\end{proof}

\begin{lem}
\label{l:isotopy_CET^2_disconnected}
There is an isotopy of $H$ relative to $\TT^1$ so that $\plex{\CC \EE(H,\TT^2)}$ is disconnected.
\end{lem}

\begin{proof}
By the previous lemma, $H$ can be isotoped to meet each face in a collection of normal arcs.    By assumption there are disks $D_+$ and $D_-$ representing vertices in different components of $\plex{\CC \EE(H,\TT^1)}$. If there are such disks that also represent vertices of $\plex{\CC \EE(H,\TT^2)}$ then we are done, as this complex is a subcomplex of $\plex{\CC \EE(H,\TT^1)}$.

After isotopy, we may assume that in a neighborhood of $\TT^1$, both $D_+$ and $D_-$ are disjoint from $\TT^2 -\TT^1$. Thus, $D_+$ and $D_-$ meet $\TT^2$ in a collection of loops, arcs, and subarcs of $\TT^1$.  We assume $D_+$ and $D_-$ were chosen so that they meet the faces of $\TT^2$ in a minimal number of curves.  The result thus follows if we show that away from $\TT^1$ they are in fact disjoint from the faces.

Let $\Delta$ be a face of $\TT^2$ that meets $D_+$. Then $H \cup D_+$ cuts up $\Delta$ into a collection of planar surfaces. At least one of these planar surfaces is a disk $C$ such that either

\begin{itemize}
	\item  $\bdy C \subset D_+$.
	\item $\bdy C =\alpha \cup \beta$, where $\alpha \subset H$ and $\beta \subset D_+$.
	\item $\bdy C=\alpha' \cup \alpha'' \cup \beta \cup \gamma$, where $\alpha', \alpha'' \subset H$, $\beta \subset D_+$, and $\gamma \subset \TT^1$. 
\end{itemize}

In all three cases we can surger $D_+$ along $C$ to produce a compressing or edge-compressing disk $D'_+$ that meets $\TT^2$ fewer times. As $D_+'$ is disjoint from $D_+$, it will be in the same component of $\plex{\CC \EE(H,\TT^1)}$, contradicting our assumption that $|D_+ \cap \TT^2|$ was minimal among all such disks.
\end{proof}

Proposition~\ref{p:CE1toCE2} now follows immediately from Lemma~\ref{l:isotopy_CET^2_disconnected}.

\section{The 3-skeleton: $\plex{\CC \EE(H,\TT^2)}$  disconnected implies $H$ is almost normal}

Recall that a surface in a triangulated 3-manifold is {\it normal} if it meets each tetrahedron in a collection of triangles and quadrilaterals and {\it almost normal} if it also has a component in a single tetrahedron that is an octagon or an unknotted annulus of total boundary length 8 or less (see Figure \ref{fig:almost_normal_in_a_3-simplex2}).

\begin{figure}[h]
   \centering
   \includegraphics[width=.9\textwidth]{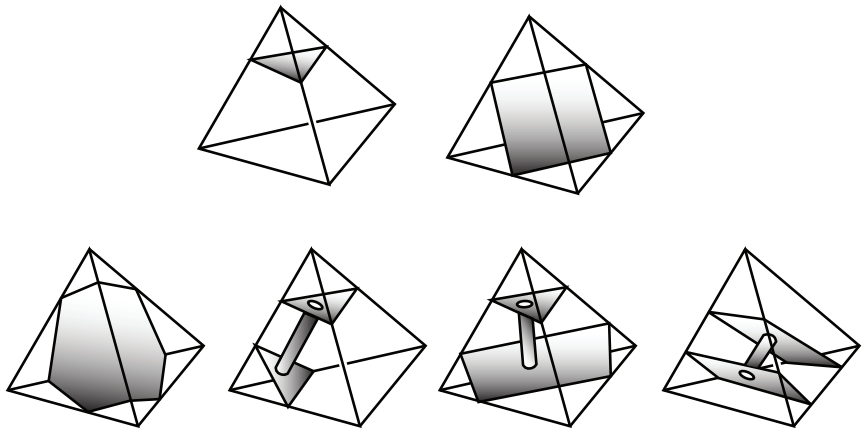}
   \caption{A normal surface meets each tetrahedron in a union of triangles and quadrilaterals, as indicated in the top two pictures. An almost normal surface has one additional component in a single tetrahedron that is either an octagon or an unknotted annulus formed from two normal disks tubed to each other, as in the bottom four pictures.}
   \label{fig:almost_normal_in_a_3-simplex2}
\end{figure}

In this section we add the final step in the proof of Theorem~\ref{t:main1} with the following:

\begin{pro}
\label{p:normalOrAlmostNormal}
The complex $\plex{\CC \EE(H,\TT^2)}$ is disconnected if and only if $H$ is almost normal.
\end{pro}

Assume that $\plex{\CC \EE(H,\TT^2)}$ is disconnected. Recall that $H$ is assumed to have no component contained in a ball. Thus, in a tetrahedron $\Delta$, every component of $H \cap \Delta$ has boundary. Moreover, by Lemma~\ref{l:each_face_normal_arcs}, $H$ meets $\partial \Delta$ in normal curves.

The combinatorics of normal curves in the boundary of a tetrahedron are well understood. Picturing a normal curve $c$ as the equator of the sphere $\bdy \Delta$, the following lemma and figure articulate the possibilities for each hemisphere of $\partial \Delta$ cut along $c$:

\begin{lem}
\label{l:tetCurve}
Let $c \subset \Delta$ be a connected normal curve pictured as the equator of $\bdy \Delta$ and  let $S_1$ and $S_2$ denote the hemispheres of $\Delta - c$. Then exactly one of the following holds:

\begin{enumerate}
\item $c$ has length 3 and one of $S_1$ or $S_2$ meets $\bdy \Delta$ as in Figure \ref{f:tetCurve}(a).   In particular $c$ is the link of a vertex of $\Delta$.
\item $c$ has length 4 and both $S_1$ and $S_2$ meet $\bdy \Delta$ as in Figure \ref{f:tetCurve}(b).  In particular, $c$ separates a pair of edges.
\item $c$ has length $4k, k \geq 2$, and both $S_1$ and  $S_2$ meet $\bdy \Delta$ as in Figure \ref{f:tetCurve}(c).  In particular, $S_i$ contains 2 vertices, three sub-edges meeting each vertex, and $2k-3$ parallel sub-edges separating the 2 vertices.
\end{enumerate}
\end{lem}

\begin{figure}[h]
               \psfrag{3}{$3$}
               \psfrag{4}{$4$}
               \psfrag{k}{$4k, k \geq 2$}
               \psfrag{a}{(a)}
               \psfrag{b}{(b)}
               \psfrag{d}{(c)}

               \psfrag{c}{$c$}
               \psfrag{n}{\tiny{($2k-3$ sub-edges)}}
               \begin{center}
                       \includegraphics[width=4 in]{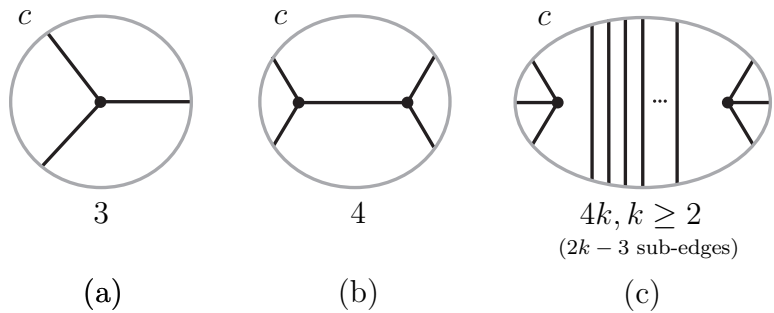}
                       \caption{Hemispheres of a tetrahedron  bounded by normal curves of lengths 3, 4, and $4k, k\geq 2$ }
                       \label{f:tetCurve}
               \end{center}
       \end{figure}

\begin{lem}
\label{l:parallel}
If two components of a normal curve have length greater than 3, then both have length $4k$ for some $k \geq 1$ and are normally parallel (i.e.~isotopic on $\bdy \Delta$ rel $\TT^1$).
\end{lem}

The above lemmas are standard results in normal surface theory
(see e.g.~\cite{Thompson}). We do not include proofs here.

We now study disks in tetrahedra whose boundaries are normal curves.  We call a disk in a tetrahedron a {\it triangle} if its boundary has length 3, a {\it quadrilateral} or just a {\it quad} if its boundary has length 4, and an {\it octagon} if its boundary has length 8.

\begin{lem}
\label{l:diskTypes}
Let $\Delta$ be a tetrahedron and assume that $H_{\Delta}$ is a properly embedded disk whose boundary is a normal curve.  Then
\begin{enumerate}
\item $\plex{\CC \EE(H_{\Delta},\TT^2)} = \emptyset \iff H_{\Delta}$ is a triangle or quadrilateral.
\item $\plex{\CC \EE(H_{\Delta},\TT^2)}$ is disconnected $\iff H_{\Delta}$  is an octagon.
\item $\plex{\CC \EE(H_{\Delta},\TT^2)}$ is connected $\iff \bdy H_{\Delta}$ has length $4k, k \geq 3$.

\end{enumerate}
\end{lem}

\begin{proof}

Note that as a disk, $H_{\Delta}$ has no compressing disks in $\Delta$. Thus we only need to focus on edge-compressing disks in $\Delta$. Since $H_{\Delta}$ has normal boundary, it has an edge-compressing disk if and only if there is some edge of $\partial \Delta$ that meets $\partial H_{\Delta}$ at least twice. Equivalently, $H_{\Delta}$ has no edge-compressing disks if and only if $H_{\Delta}$ meets each edge of $\partial \Delta$ at most once. Hence, according to Figure \ref{f:tetCurve}, the case that $H_{\Delta}$ has no edge-compressing disks is equivalent to $H_{\Delta}$ being a triangle or a quad. This addresses case (1).

Now consider the case that $H_{\Delta}$ has an edge-compressing disk. This is true if and only if there is a sub-edge $\beta \subset \bdy \Delta$ meeting $\bdy H_{\Delta}$ exactly twice. Consider Figure \ref{f:tetCurve}. This occurs only if $\bdy H_{\Delta}$ has length $4k, k \geq 2$, and in its hemisphere $\beta$ is one of the $2k-3$ vertical arcs pictured. Note that $\beta$ splits its hemisphere into two disks, each of which can be tilted slightly into $\Delta$ to yield edge-compressing disks $E_\beta$ and $E_\beta'$ for $H_{\Delta}$. The same is true for any subarcs in the other hemisphere of $\partial \Delta$ cut along $\partial H$ (such subarcs must exist by conclusion (3) of Lemma~\ref{l:tetCurve}). In particular, $H_{\Delta}$ has edge-compressing disks on both sides.

If $k=2$, choose edge-compressing disks $E_1$ and $E_2$ for opposite sides of $H_{\Delta}$. Each must run along the sole vertical arc for its hemisphere, namely $\beta$ and the arc in the other hemisphere, call it $\gamma$. Note that $\beta$ does not meet $\gamma$ in an endpoint, for otherwise inspecting Figure~\ref{f:tetCurve} in the case $k=2$ reveals that $\beta \cup \gamma$ would form a cycle on $\partial \Delta$. Thus each end of $\beta$ and $\gamma$ connects to a different tripod of the other hemisphere, the endpoints of $\beta \subset E_1$ and $\gamma \subset E_2$ alternate on $\bdy H_{\Delta}$. It follows that $E_1$ and $E_2$ meet on $H_{\Delta}$. As any edge-compressing disks for $H_{\Delta}$ must meet $\partial \Delta$ in $\beta$ and $\gamma$, it follows that $\plex{\CC \EE(H,\TT^2)}$ is disconnected.

If $k \geq 3$, then the arc $\beta$ is not the only vertical arc in its hemisphere of $\partial \Delta$. In particular, there is an arc $\gamma$ in the other hemisphere of $\partial \Delta$ that meets $\beta$ in a single endpoint. The edge-compressing disk $E_{\beta}$ and one of the edge-compressing disks $E_{\gamma}$ or $E_{\gamma'}$ can thus be made disjoint in a neighborhood of the 1-skeleton. This implies $\plex{\CC \EE(H_{\Delta},\TT^2)}$ is connected. This completes cases (2) and (3).
\end{proof}

We now consider the entire surface $H$ in $M$. 

\begin{lem}
\label{l:disconnected}
If $\plex{\CC \EE(H,\TT^2)}$ is disconnected, then for exactly one tetrahedron $\Delta$, $\plex{\CC \EE(H \cap \Delta,\TT^2)}$ is disconnected and for any other tetrahedron $\Delta'$, $\plex{\CC \EE(H \cap \Delta',\TT^2)} = \emptyset$.
\end{lem}

\begin{proof}
If $\plex{\CC \EE(H,\TT^2)}$ is disconnected, there must be compressing or edge-compressing disks $D$ and $E$ on opposite sides of $H$ that meet. Therefore, both $D$ and $E$ lie in the same tetrahedron, $\Delta$. Suppose $\Delta'$ is some other tetrahedron for which  $\plex{\CC \EE(H \cap \Delta',\TT^2)}$ is non-empty. Let $D'$ be a compressing or edge-compressing disk for $H \cap \Delta'$, and assume without loss of generality that $D'$ is on the same side of $H$ as $D$. Then, as $E$ and $D'$ are in different tetrahedra, they are disjoint. It follow from Lemma~\ref{l:OneSideConnected} that $\plex{\CC \EE(H,\TT^2)}$ is connected.
\end{proof}

We will deal with the latter case of Lemma~\ref{l:disconnected} first. 

\begin{lem}
\label{l:empty_in_tetrahedron}
Let $H$ be a surface and $\Delta'$ a tetrahedron. Then, the complex $\plex{\CC \EE(H \cap \Delta',\TT^2)} = \emptyset$ if and only if each component of $H \cap \Delta'$ is a normal triangle or quadrilateral.
\end{lem}

\begin{proof}
($\Leftarrow$)
If each component of $H \cap \Delta'$ is a disk, then $\plex{\CC \EE(H \cap \Delta',\TT^2)}$ has no vertices corresponding to compressing disks. Moreover, no component of $H \cap \Delta'$ can meet an edge of $\bdy \Delta'$ more than once (the first two cases in Figure \ref{f:tetCurve}). Hence, $\plex{\CC \EE(H \cap \Delta',\TT^2)}$ has no vertices corresponding to edge compressions either.

($\Rightarrow$) If $\plex{\CC \EE(H \cap \Delta',\TT^2)}=\emptyset$ then $H \cap \Delta'$ must consist of disks. If $H \cap \Delta'$ has a component whose boundary (a normal curve) has length greater than $4$, then an outermost such component contains an edge-compressing disk that can be tilted past all triangles, contradicting $\plex{\CC \EE(H \cap \Delta',\TT^2)}=\emptyset$.
\end{proof}

Now consider the case of Lemma~\ref{l:disconnected} in which $H \cap \Delta$ has disconnected disk complex in the tetrahedron $\Delta$. We constrain the topological type of $H$ in $\Delta$ via the following:

\begin{lem}
\label{l:almostNormal}
If $\plex{\CC \EE(H \cap \Delta,\TT^2)}$ is disconnected, then $H \cap \Delta$ is the union of (perhaps an empty set of) triangles and quadrilaterals, and exactly one exceptional piece which is an octagon or an unknotted annulus whose boundary curves each have either length 3 or 4.
\end{lem}

\begin{proof}
By Lemma \ref{l:parallel} we can write $H \cap \Delta$ as the disjoint union
$$H = D_3 \cup D_{4k} \cup A$$
where $D_{\ell}$ is a union of disks whose boundaries have normal length $\ell$ and $A$ is the optimistically labelled union of all non-disk surfaces.
\medskip

{\it Claim. If $A=\emptyset$, then $D_{4k}$ is an octagon and the lemma holds.} 
\smallskip

If $A=\emptyset$ then $H$ is incompressible, and so there must be a pair of edge-compressing disks $E$ and $E'$ for opposite sides of $H$.  Moreover, $E$ and $E'$ meet, so it must be that they are both edge-compressing disks for the same disk component $D \subset D_{4k}$.  By Lemma \ref{l:diskTypes}, it is the case that $k \geq 2$ and $\bdy D$ bounds two hemispheres as in Figure \ref{f:tetCurve}, where $E$ and $E'$ split each hemisphere in two.  Triangles are the only disks disjoint from $D \cup E \cup E'$, so $D$ must be the sole component $D_{4k}=D$. Note that any edge-compressing disk for $H \cap \Delta$ must thus be an edge-compressing disk for $D$, and conversely, any edge-compressing disk for $D$ can be pushed past the elements of $D_3$ to an edge-compressing disk for $H \cap \Delta$. Thus $\plex{\CC \EE(H \cap \Delta, \TT^2)}=\plex{\CC \EE(D, \TT^2)}$ is disconnected. By Lemma \ref{l:diskTypes}, $D$ is an octagon. This proves the claim.
\medskip

We proceed with the assumption that $A \neq \emptyset$.

\medskip

{\it Claim. If $A \neq \emptyset$, then $A$ is an unknotted annulus.} 
\smallskip

Let $b_1, b_2, \ldots b_n$ denote all components of $\bdy A$ that are innermost on $\bdy \Delta$. Then $b_i$ bounds a disk in $\bdy \Delta$. Let $B_i$ denote this disk after pushing its interior into $\Delta$ until it is disjoint from the disk collection $D_3 \cup D_{4k}$. Since $\bdy B_i = b_i$ bounds a disk in $\bdy \Delta$, it can be isotoped to be disjoint from any compressing disk for the other side of $A$ (which happens to be the only component of $H \cap \Delta$ that can have compressing disks).  It follows that $A$ must have an edge-compressing disk $E$ on the other side that meets every $b_i$. But $E \cap \partial \Delta$ has only two endpoints, thus $n \leq 2$. If $n=1$, then $A$ has only one boundary component, $b_1$, and it bounds disks on {\it both} sides of $A$, which can be isotoped to be disjoint from each other and from $D_3 \cup D_{4k}$. This contradicts the fact that $\plex{\CC \EE(H,\TT^2)}$ is disconnected. Thus, $n=2$.

Let $N(E)$ denote a neighborhood of $E$ in $\Delta$, and $\bdy _{fr} N(E)$ the frontier of $N(E)$ in $\Delta$. Let $R$ denote the component of $\bdy _{fr} N(E)-(B_1 \cup B_2)$ that meets both $B_1$ and $B_2$. Finally, consider the disk $(B_1 \cup B_2)-N(E) \cup R$ (see Figure~\ref{f:band_sum}). This disk is on the same side of $A$ as $B_1$ and $B_2$, but disjoint from $E$. To avoid a contradiction, we conclude that the banded disk is a trivial disk, i.e., $b_1$ and $b_2$ co-bound an unknotted annulus. Moreover, $A$ does not contain a component distinct from this annulus, for any other component would be forced to have boundary disjoint from $\bdy (A \cup E)$, implying $\bdy A$ had more than 2 innermost curves. This completes the proof of the claim.

\begin{figure}[h]

\psfrag{1}{$B_1$}
\psfrag{2}{$B_2$}
\psfrag{e}{$E$}
                       \begin{center}
                       \includegraphics[width=3.5 in]{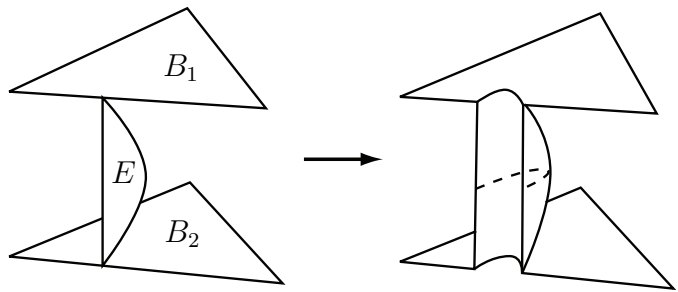}
                       \caption{The disk $(B_1 \cup B_2)-N(E) \cup R$.}
                       \label{f:band_sum}
               \end{center}
       \end{figure}

\medskip

Now it is a matter of restricting the lengths of $\partial B_1$ and $\partial B_2$ as normal curves. This is established with the final claim.

\medskip

{\it Claim: Each component of $\bdy A$ has length at most 4.}
\smallskip

Consider the hemisphere $S_1$ of $\bdy \Delta$ bounded by $b_1$ that does not contain $b_2$. If $|b_1|>4$ then $S_1$ contains a subarc $\alpha$ of an edge, as in Figure \ref{f:tetCurve}. The arc $\alpha$ divides $S_1$ into two disks, and both can be tilted into $\Delta$ (and pushed past $D_3$), to form edge-compressing disks $E_1$ and $E_2$ for $H \cap \Delta$. Now note that $E \cap A$ is an arc that connects $b_1$ to $b_2$, whereas both $e_1=E_1 \cap A$ and $e_2=E_2 \cap A$ connect two points of $b_1$. Furthermore, $e_1 \cup e_2$ is isotopic to a core circle of $A$, and can thus be isotoped to meet $E \cap A$ exactly once. It follows that one of $e_1$ or $e_2$ is disjoint from $E \cap A$. Without loss of generality, assume the former. Thus $E_1$ and $E$ are disjoint edge-compressing disks for $H \cap \Delta$ on opposite sides. By Lemma \ref{l:OneSideConnected}, this violates our assumption that $\plex{\CC \EE(H \cap \Delta,\TT^2)}$ is disconnected.
\end{proof}

Proposition~\ref{p:normalOrAlmostNormal} now immediately follows. The forward direction is implied by Lemmas \ref{l:disconnected} and \ref{l:almostNormal}, and the backward direction is a straightforward observation. This completes the final step in the proof of Theorem~\ref{t:main1}.

\bibliographystyle{alpha}
\bibliography{Almost_normal_bibliography}

\end{document}